\documentclass[12pt,leqno]{article}
\usepackage{amsmath,amssymb,amsthm, epsfig}
\usepackage{amsfonts}
\usepackage{fontenc}
\usepackage{ulem}
\usepackage[mathscr]{eucal}

\usepackage{mathptmx}

\title{On the $L_r$-operators penalized by $(r+1)$-mean curvature}

\author{ \textsc{L. I. S. Souza}\footnote{This work was partially supported by UFC} 
}

\newlength{\hchng}
\newlength{\vchng}
\def \R {\mathbb{R}}

\newtheorem{theorem}{Theorem}[section]
\newtheorem{lemma}[theorem]{Lemma}

\newtheorem{corollary}[theorem]{Corollary}

\theoremstyle{definition}

\newtheorem{remark}[theorem]{Remark}
\numberwithin{equation}{section}

\newcommand{\vertiii}[1]{{\left\vert\kern-0.25ex\left\vert\kern-0.25ex\left\vert #1 \right\vert\kern-0.25ex\right\vert\kern-0.25ex\right\vert}}

\newcommand{\intav}[1]{\mathchoice {\mathop{\vrule width 6pt height 3 pt depth  -2.5pt
\kern -8pt \intop}\nolimits_{\kern -6pt#1}} {\mathop{\vrule width
5pt height 3  pt depth -2.6pt \kern -6pt \intop}\nolimits_{#1}}
{\mathop{\vrule width 5pt height 3 pt depth -2.6pt \kern -6pt
\intop}\nolimits_{#1}} {\mathop{\vrule width 5pt height 3 pt depth
-2.6pt \kern -6pt \intop}\nolimits_{#1}}}

\newcommand{\Rmuv}{(-L_r +\mu)^{-1}}

\newcommand{\Rm}{R_{\mu}}
\newcommand{\Rg}{R_{\gamma}}

\newcommand{\roa}{W_r}

\newcommand{\Km}{K_{\mu}}
\newcommand{\Rz}{R_0}

\newcommand{\lambdad}{\lambda_1(-L_r)}
\newcommand{\dv}{\mathrm{div}}

\newcommand{\sA}{\textbf{A}}
\newcommand{\Lr}{{\mathcal{L}_r}}

\date{\today}

\begin{document}
\maketitle

\begin{abstract}
In this paper, we establish the non-positivity of the second eigenvalue of the Schr\"odinger operator $-\textrm{div}\big( P_r \nabla\cdot\big) - W_r^2$ on a closed hypersurface $\Sigma^n$ of $\mathbb{R}^{n+1}$, where $W_r$ is a power of the $(r+1)$-th mean curvature of $\Sigma^n$. In the case that this eigenvalue is null we have a characterization of the sphere. This generalizes a result of Evans and Loss proved for the Laplace-Beltrame operator penalized by the square of the mean curvature.

\end{abstract}


\section{Introduction}\label{sec:1}
In 1997, Evans and Loss \cite{Evans} obtained the following rigidity result:

\begin{theorem}Let $\Omega$ a smooth compact oriented hypersurface of dimension $d$ immersed in $R^{d+1}$; in particular self-intersections are allowed. The metric on that surface is the standard Euclidean metric inherited from $R^{d+1}$. Then the second eigenvalue $\lambda_2$ of the operator
$$H=-\Delta-\frac{1}{d}h^2$$
is strictly negative unless $\Omega$ is a sphere, in which case $\lambda_2$ equals zero.
\end{theorem}

In particular, when $d=2$ the previous result gives a proof for a conjecture of Alikakos and Fusco about hyper surfaces embedded in $\mathbb{R}^3$ (cf. \cite{Evans}). 

The goal of this paper is to extend this result for a more general class of elliptic geometric operators. In order to state our main result, we need to introduce a few definitions and notation. Let $\phi\colon M^n\to \overline{M}^{n+1}$ be an isometric immersion, and denote by ${\bf A}$ the second fundamental form associated to $\phi$. It is known that ${\bf A}$ has $n$-geometric invariants. They are given by the elementary symmetric functions $S_r$ of the principal curvatures $\kappa_1,\dots, \kappa_n$ as follows:
\[
S_r:= \sum_{i_1<\dots< i_r}\kappa_{i_1}\dots \kappa_{i_r}\quad (1\leq r\leq n). 
\]
The $r$-curvature $H_r$ of $\phi$ is then defined by
$$
H_r:=\frac{S_r}{{n \choose r}}.
$$
Notice that $H_1$ corresponds to the mean curvature and $H_n$ the Gauss-Kronecker curvature of $\phi$. The Newton's transformations of $\phi$ are the operators $P_r$ defined inductively by
\[
P_0=I,
\]
\[
P_r=S_r I - {\bf A} P_{r-1}.
\]
The so-called $L_r$-operators are defined by $L_r:=\textrm{div}\big(P_r \nabla\cdot\big)$. It is known that if every $H_r$ is positive, then $L_r$ is elliptic. Let $c_r=(n-r){{n \choose r}}$ and define the potential $W_r=\big( c_r \,H_{r+1}^{\frac{r+2}{r+1}}\big)^{1/2}$. In what follows we consider the following class of divergence operators 
\[
\mathcal{L}_r:= -L_r - W^2_r. 
\]
We are now able to state the main result of this paper. 
\begin{theorem}\label{thm:M} Let $\Sigma^n$ be a $n$-dimensional closed hypersurface embedded in $\mathbb{R}^{n+1}$. Assume that $H_{r+1}>0$. Then 
\[
\lambda_2(\mathcal{L}_r)\leq 0,
\]
with equality occurring if and only if $\Sigma^n=\mathbb{S}^n$. 
\end{theorem}

The proof is based upon the following principle:
\begin{lemma}[Birman Schwinger's Principle] 
Let $L=\dv(A(x)\nabla .)$ be, where $A(x)$ is a matrix uniformly elliptic, $L:H^2(\Omega)\to L^2(\Omega)$, where $\Omega$ is a bounded domain.

Consider the self-adjoint operator $-L-W^2(x)$, where $W^2$ is relatively bounded with respect to $-L$ 
(i.e. $Dom(-L)\subset Dom(W^2)$ and exists constants $a,b\ge 0$, such that 
$${\|W^2u\|_2\le a\|u\|_2+b\|-Lu\|_2},\;\;\mathit{for}\; \mathit{all}\;{u\in Dom(-L)}).$$

A number $-\mu<0$ is eigenvalue of $-L-W^2$, if only if $1$ is eigenvalue of the bounded positive operator $$K_\mu:=W(-L+\mu)^{-1}W$$
\end{lemma}

This result can be obtained as a corollary of a more general principle demonstrated by Klaus in the paper \cite{Martin}.

\section{Proof of Theorem \ref{thm:M}}

For the above proof, the following lemma will be used:

\begin{lemma}
Let $\Sigma^n$ be a $n$-dimensional closed hypersurface embedded in $\mathbb{R}^{n+1}$ and
consider the operator ${\Lr=-L_r-W_r^2}$. Suppose there $f\in L^2(\Sigma)$ satisfying:
\\[0,2cm]$1)\int_{\Sigma}f\roa d\Sigma=0$;
\\[0,2cm]$2)\langle R_0(\roa f),\roa f\rangle>\|f\|_2^2$.

Then the operator $\Lr$ has two negative eigenvalues
\end{lemma}

\begin{proof}
The proof is herein presented in three following steps:

 Recalling that for $\mu>0$ the resolvent operator $(-L_r+\mu)^{-1}$ is a bounded operator in $L^2(\Sigma)$ and $(-L_r)^{-1}$,  therefore, it is defined on the set of functions with zero mean.
\\[0,1cm]
\textbf{Step 1}: For $g\in L^2(\Sigma)$ with $\int_{\Sigma} g\hspace{0,05cm}d\Sigma=0$, we have
$$
\lim_{\mu\to 0}\|(-L_r)^{-1}g-(-L_r+\mu)^{-1}g\|_2=0 
$$
In fact, let
\begin{equation}
(-L_r+\mu)^{-1}g=\varphi, \label{eq1}
\end{equation} 
so that,  $$-L_r\varphi+\mu\varphi=g.$$ Therefore $\int_{\Sigma}\varphi d\Sigma=0$ because $\int_{\Sigma}L_r\varphi d\Sigma=0$. The latter follows from divergence theorem. By \eqref{eq1}, we get 
\begin{equation}
-L_r\varphi+\mu\varphi=g.\label{eq2}
\end{equation}
Applying $(-L_r)^{-1}$ in equation \eqref{eq2}, we obtain that $$\varphi+\mu(-L_r)^{-1}\varphi=(-L_r)^{-1}g$$ and therefore
\begin{equation}
\|(-L_r)^{-1}g-(-L_r+\mu)^{-1}g\|_2=\|\mu(-L_r)^{-1}f\|_2$$ $$\hspace{5,5cm}\le\mu\|{-L_r}^{-1}\|\|f\|_2$$ $$\hspace{7,5cm}=\mu\|{-L_r}^{-1}\|\|(-L_r+\mu)^{-1}g\|_2\label{eq3}
\end{equation}
Now in order to estimate the norm $\|(-L_r+\mu)^{-1}g\|_2$, we first multiply both sides of \eqref{eq2} for $\varphi$,
then apply divergence theorem to get
\begin{equation}
\int_{\Sigma}\langle P_r\nabla\varphi,\nabla\varphi\rangle d\Sigma+\mu\int_{\Sigma}\varphi^2d\Sigma=\int_{\Sigma} \varphi gd\Sigma.\label{eq4}
\end{equation}
Since $\varphi$ has zero mean, we can use Rayleigh's principle to deduce
$$\lambda_1(-L_r)\int_{\Sigma}\varphi^2d\Sigma\le\int_{\Sigma}\langle P_r\nabla\varphi,\nabla\varphi\rangle d\Sigma$$
On the other hand, Cauchy's inequality yields 
$$\int_{\Sigma}\varphi gd\Sigma\le\left(\frac{1}{4\varepsilon}\int_{\Sigma} g^2d\Sigma+\varepsilon\int_{\Sigma}\varphi^2d\Sigma\right)$$
Thus for $\varepsilon=({\lambdad}+\mu)/2$ we have $$\int_{\Sigma}\varphi^2d\Sigma\le\frac{1}{({\lambdad}+\mu)^2}\int_{\Sigma} g^2d\Sigma.$$
Consequently,
$$\|\Rm g\|_2\le\frac{1}{(\lambdad+\mu)}\|g\|_2.$$
From the estimate above, it follows that $$\|\Rm g-\Rg g\|_2\le\frac{|\mu-\gamma|\|g\|_2}{(\lambdad+\mu)(\lambdad+\gamma)}$$
Set $\Km:=\roa\Rmuv\roa$, now for $\mu$ positive and close to zero, $\Km$ has an eigenvalue greater than 1.

\textbf{Step 2}: Exist $-\mu_1<0$,such that $\|K_{\mu_1}|_{[\roa]^{\perp}}\|>1.$

The operator  $\Km|_{[\roa]^{\perp}}$ is compact, simetric and positive, soon $\|\Km|_{{[\roa]}^{\perp}}\|$ is an eigenvalue of $\Km|_{[\roa]^{\perp}}$.

Since,
$$\langle R_0(\roa f),\roa f\rangle>\|f\|_2^2,$$
and $\Km\to K_0$ in $\mathcal{B}([\roa]^{\perp})$, when $\mu\to 0$ with $K_0=\roa\Rz\roa$, so we have $\|K_0|_{[\roa]^{\perp}}\|>1$, thus, there exists $-\mu_1<0$ such as
$\|K_{\mu_1}|_{[\roa]^{\perp}}\|>1$.

\textbf{Step 3}: The Step 2 implies Lemma.

As $K _{\mu}$  is positive, we have $\|K_{\mu}\|$ which is the largest eigenvalue of $K_{\mu}$. Furthermore,
\begin{equation}
\label{Kmu}
\|K_{\mu}\|\leq \frac{1}{\lambda_1(-L_r|_{H^2(\Sigma)\cap[1]^{\perp}})+\mu}\|\roa\|_{\infty}^2.
\end{equation}
 Thus, the eigenvalue $\|K_{\mu}\|$ goes to zero when $\mu$ as to infinity. Particulary, there is $-\mu_2<0$ such as
$\|K_{\mu_2}\|<1$.

Hence, we show that there exist $\mu_2$ and $\mu_1$ constants, such that $$\|K_{\mu_2}\|<1<\|K_{\mu_1}\|,$$ and $\mu\mapsto\|K_\mu\|$ is continuous, we have by Intermediate Value Theorem, there is $-\mu_0$ such that $\|K_{\mu_0}\|=1$.

The Birman Schwinger principle,
$-\mu_0<0$ is eigenvalue of $\overline{\Lr}=\Lr|_{[1]^{\perp}}$, i.e., there is a nonzero function $f\in H^2(\Sigma)\cap [1]^{\perp}$
such that $\overline{\Lr}f =-\mu_0 f $. Naturally $-\mu_0$ is also eigenvalue of the operator $\Lr$.

Suppose by contradiction, that $-\mu_0$ is the only negative eigenvalue of $\Lr$.

In this case $-\mu_0$ would be the first eigenvalue with a
first self-space given by $[f]=\{c \, f; \ c\in \mathbb{R} \}$ and $\Lr$ restricted to the subspace $[f]^{\perp}$ would be a positive element. On the flip side we have $f\in [1]^{\perp}$.

Thus, the constant function $1\in [f]^{\perp}$, implies $\langle\Lr 1,1\rangle_2 \geq 0$. It is a contradiction.

Hence, the operator $\Lr$ has more than one negative eigenvalue, if there is ${f\in L^2(\Sigma)}$ satisfying $1)$ and $2)$. 
\end{proof}

Now let's proof of Theorem 1.2.

\begin{proof}
Let $\phi:\Sigma^n\to\R^{n+1}$ be an isometric immersion, by \cite{Hilario}, we have the following equation satisfied:
\begin{equation}
-L_r\phi=c_rH_{r+1}N,
\end{equation}
where $N$ is the normal vector of the surface.

Thus each coordinate satisfies $-L_r\phi_i=c_rH_{r+1}N_i$ , with $i\in\{1,...,n+1\}$.

Denote by $(\phi_i)_{\Sigma}:=\frac{1}{vol.\Sigma}\int_\Sigma \phi_id\Sigma$, and $(\phi)_{\Sigma}:=((\phi_1)_\Sigma,...,(\phi_{n+1})_\Sigma)$.

Choosing $f_i$ so that $$f_iW_r=c_rH_{r+1}N_i,$$ we have $$f_i=(c_rH_{r+1}^{\frac{r}{r+1}})^{\frac{1}{2}}N_i.$$
Observing that 
$$\Rz(W_rf_i)=\Rz(c_rH_{r+1}N_i)=\Rz(-L_r(\phi_i-(\phi_i)_\Sigma))=\phi_i-(\phi_i)_\Sigma.$$
By multiplying both sides, equal to $\phi_i-(\phi_i)_\Sigma$ and using Divergence Theorem, we conclude that
$$\langle\Rz(W_rf_i),W_rf_i\rangle_2=\langle P_r\nabla\phi_i,\nabla\phi_i\rangle_2=\int_\Sigma c_rH_{r+1}(\phi_i-(\phi_i)_\Sigma)N_i d\Sigma.$$

Summing up both sides with $i$ vanging from $1$ to $n+1$, we have  

$$\sum\limits_{i=1}^{n+1}\langle\Rz(W_rf_i),W_rf_i\rangle_2=\sum\limits_{i=1}^{n+1}\langle P_r\nabla\phi_i,\nabla\phi_i\rangle_2=\int_\Sigma c_rH_{r+1}\langle\phi-(\phi)_\Sigma,N\rangle d\Sigma.$$

 In \cite{Hilario}, we know from Minkowski's integral formula

$$\int_\Sigma H_rd\Sigma-\int_\Sigma H_{r+1}\langle\phi-(\phi)_\Sigma,N\rangle d\Sigma=0.$$

 Thus, replacing the previous expression, we have
 $$\sum\limits_{i=1}^{n+1}\langle\Rz(W_rf_i),W_rf_i\rangle_2=\sum\limits_{i=1}^{n+1}\langle P_r\nabla\phi_i,\nabla\phi_i\rangle_2=\int_\Sigma c_r H_rd\Sigma.$$
 
  By \cite{Hilario} using the classical inequality $H_r^{\frac{1}{r}}\ge H_{r+1}^{\frac{1}{r+1}}$, for $r\ge 1$, we have

$$\sum\limits_{i=1}^{n+1}\langle\Rz(W_rf_i),W_rf_i\rangle_2=\int_\Sigma c_r H_rd\Sigma\ge \int_\Sigma c_r H_{r+1}^{\frac{r}{r+1}}d\Sigma=\sum\limits_{i=1}^{n+1}\int_\Sigma c_r H_{r+1}^{\frac{r}{r+1}}N_i^2d\Sigma=\sum\limits_{i=1}^{n+1}\|f_i\|_2^2.$$

\begin{remark}
If $r=0$, we have written the sums above being identical and the only step that does not appear is the gap between the bends, however it is easy to see that the rest of the argument is following analogous to other cases.
\end{remark}

Define $d_i=\langle\Rz(W_rf_i),W_rf_i\rangle_2-\|f_i\|_2^2$,thus $\sum\limits_{i=1}^{n+1}d_i\ge 0$ and then two possibilities may occur:

1) There is $i\in\{1, ..., n + 1 \}$ such that $d_i>0$;

2) $d_i=0$, for all $i\in\{1,...,n+1\}$.

If 1) occurs, we have $f_i$ being the Lemma conditions and therefore
$$\lambda_2(\Lr)<0.$$

 If 2) occurs, we have all the $d_i$  void. For this we use Lagrange multipliers.
\vspace{0,3cm}

 Now consider the functionals $\Psi , \Phi:L^2(\Sigma)\to \mathbb{R}$ given by
 
$$\Psi(f)=\langle R_0(W_r f),W_r f \rangle-\|f\|_2^2,\;\;\;\;\;\;\Phi(f)=\langle W_r,f \rangle_2$$
 and the set of constraints
$$
S=\{f\in L^2(\Sigma);\;\Phi(f)=\langle W_r,f \rangle_2=0 \}.
$$
 We have to study two possibilities: 

\noindent 1) $\inf\{\Psi(f);\; f\in S \}<0 $ or

\noindent 2) $\inf\{\Psi(f);\; f\in S \} = 0$, since $0\in S$ and $\Psi (0)=0$.

  In the first case, will be the function $f\in S$ such that $\Psi(f)<0$.

So $f$ is a critical function for $\Psi$ on $S$ and then the method of Lagrange multipliers have to exist $\Gamma\in\R$, such that
$$\Psi'(f)=\Gamma\Phi'(f)$$
which resulted in the following Euler-Lagrange equation
$$W_r\Rz(W_rf)-f=\Gamma W_r.$$
 Multiplying both sides of the above equation for $f\in S$ and integrating, we have

$$0=\Gamma\langle  W_r,f\rangle=\langle R_0(W_r f),W_r f \rangle-\|f\|_2^2<0.$$
Thus the contradiction. This is the case \noindent 1) not occuring.
In the second case, we have seen that each $f_i\in S$ and $\Psi(f_i)={\inf\{\Psi(f);\;{f\in S}\}=0}$. By the Method of Lagrange Multipliers, there exists $\Gamma \in\mathbb{R}$ such that ${\Psi'(f_i)=\Gamma \Phi'(f_i)}$.
Hence, we obtain that each $f_i$ satisfies the following Euler-Lagrange equation,
$$W_r\Rz(W_rf_i)=f_i+\Gamma W_r,$$ therefore we conclude that
$$W_r(\Rz(W_rf_i)-\Gamma)=f_i,$$
$$W_r(\phi_i-(\phi_i)_\Sigma-\Gamma)=f_i,$$
then 
$$\phi_i-(\phi_i)_\Sigma-\Gamma=\frac{f_i}{W_r}=H_{r+1}^{\frac{1}{r+1}}N_i$$

 Thus, have its version vector

$$\phi-(\phi)_\Sigma-\Gamma=H_{r+1}^{\frac{1}{r+1}}N.$$

 Differentiating the above expression along any curve $\Sigma$, we conclude that the derivative of $H_{r+1}^{\frac{1}{r+1}}$ is zero, so $H_{r+1}$ is constant, then $\Sigma^n=\mathbb{S}^n$ by \cite{Ros}.

 In fact in this case we have $\lambda_2(\Lr)=0$, as we have
$$W_r(\phi_i-(\phi_i)_\Sigma-\Gamma)=f_i,$$
and multiplying both sides by the expression $W_r$, we obtain
$$W_r^2(\phi_i-(\phi_i)_\Sigma-\Gamma)=W_rf_i=-L_r(\phi_i-(\phi_i)_\Sigma-\Gamma),$$
thus $\psi=\phi_i-(\phi_i)_\Sigma-\Gamma$ is the second eigenfunction of $\Lr=-L_r-W_r^2$, and $\Lr\psi=0$. 

Define the operator $T_r=-L_r-c_r\|\sA\|^{r+2}$.

\begin{corollary}
Under the same conditions of Theorem 2, $\lambda_2(T_r)\le 0$ with equality if and only if $\Sigma^n=\mathbb{S}^n$.
\end{corollary}

The proof of the corollary follows immediately from the Jensen's inequality and the min-max principle. This finishes the proof.

\end{proof}


\vskip1truecm

\scshape

\noindent L. I. S. Souza

\noindent Universidade Federal do Cear\'a

\noindent Departmento de Matem\'atica

\noindent Av. Humberto Monte S/N, 60455-760, Bl 914

\noindent E-mail address: \qquad {\tt leoivo@alu.ufc.br}

\scshape
\bigskip

\vskip 1truecm

\scshape

\end{document}